\documentclass[11pt,british,refpage,intoc,bibliography=totoc,index=totoc,BCOR=7.5mm,captions=tableheading]{extarticle}
\usepackage{mathptmx}
\usepackage[T1]{fontenc}
\usepackage[latin9]{inputenc}
\usepackage[a4paper]{geometry}
\usepackage{color}
\usepackage{babel}
\usepackage{mathrsfs}
\usepackage{amsmath}
\usepackage{amsthm}
\usepackage{amssymb}
\usepackage[numbers]{natbib}
\usepackage[unicode=true,pdfusetitle,
 bookmarks=true,bookmarksnumbered=true,bookmarksopen=false,
 breaklinks=false,pdfborder={0 0 1},backref=false,colorlinks=true]
 {hyperref}
\hypersetup{
 pdfborderstyle=,linkcolor=black,citecolor=blue,urlcolor=black,filecolor=blue,pdfpagelayout=OneColumn,pdfnewwindow=true,pdfstartview=XYZ,plainpages=false}

\makeatletter
\numberwithin{equation}{section}
\theoremstyle{plain}
\newtheorem{thm}{\protect\theoremname}[section]
\theoremstyle{remark}
\newtheorem{rem}[thm]{\protect\remarkname}
\theoremstyle{plain}
\newtheorem{lem}[thm]{\protect\lemmaname}

\@ifundefined{date}{}{\date{}}
\usepackage{babel}
\usepackage{caption}
\usepackage[nottoc]{tocbibind}
\allowdisplaybreaks[4]
\usepackage{dsfont}
\DeclareMathAlphabet{\mathcal}{OMS}{cmsy}{m}{n}

\providecommand{\lemmaname}{Lemma}
\providecommand{\remarkname}{Remark}
\providecommand{\theoremname}{Theorem}

\makeatother

\providecommand{\lemmaname}{Lemma}
\providecommand{\remarkname}{Remark}
\providecommand{\theoremname}{Theorem}

\begin{document}
\title{On vorticity and expansion-rate of fluid flows,\\
 conditional law duality and their representations}
\author{Zhongmin Qian\thanks{Mathematical Institute, University of Oxford, Oxford OX2 6GG, and
Oxford Suzhou Centre for Advanced Research, Suzhou, China. Email:
$\mathtt{qianz@maths.ox.ac.uk}$}\ \ and\ Zihao Shen\thanks{Mathematical Institute, University of Oxford, Oxford OX2 6GG. Email:
$\mathtt{zihao.shen@bnc.ox.ac.uk}$} }
\maketitle
\begin{abstract}
By using a formulation of motion equations for a viscous (compressible)
fluid flow in terms of the vorticity and the rate of expansion as
the main fluid dynamical variables, an approximation model is established
for compressible flows with slowly varied (over the space) fluid density.
The probabilistic tools and the main ingredient such as the duality
of conditional laws and the forward type Feynman-Kac formula are established
for elliptic operators of second order, in order to formulate the
corresponding random vortex method for a class of viscous compressible
fluid flows, based on their approximation motion equations.

\medskip{}


\medskip{}

\end{abstract}

\section{Introduction}

It is known that the vortex formulation for fluid dynamics is very
useful in the study of incompressible fluid flows. This is particularly
true for the purpose of numerically calculating solutions of the motion
equations for (incompressible) fluid flows and turbulent flows. The
vortex techniques have been used in the study of compressible fluid
flows too, cf. \citep{Lions1998} for example. The motion equations
for a viscous fluid flow are known (cf. Landau and Lifschitz \citep{Landau-LifshitzFluids}
for example) which are formulated as a system of partial differential
equations (PDEs) that the flow velocity $u$, the fluid density $\rho$,
and the pressure $p$ have to satisfy. For a viscous (compressible)
fluid flow, these PDEs need to be complemented by a state equation
from the Kinetic theory. There is a huge literature on the study of
various mathematical models of compressible fluid flows, which are
in one or another sense approximation models of the motion equations
for general fluid flows to capture features of the flows in question,
see for example Chen and Feldman \citep{ChenFeldman2010}, \citep{ChenFeldman2018}
and etc. on potential flows, Dafermos \citep{Dafermos2005} on conservation
law models, and Lions \citep{Lions1998} on general aspects of mathematics
in the study of fluid dynamics, as a small sample. Most of the mathematical
models for compressible fluid flows in literature are obtained by
specifying a state equation and/or by specifying a form of the velocity
(rotational or not), while there are exceptional and important mathematical
models, the Navier-Stokes equations for incompressible fluid flows,
where a state equation is not needed. Indeed the continuity equation
and the momentum equation are sufficient to describe an incompressible
fluid flow (at least for small time).

The goal of the present work attempts to extend the random vortex
method developed for incompressible flows to compressible viscous
flows from a prospective of their numerical study, numerical experiments
however are not included in the present paper. To this end we examine
the formulation of the motion equations for a viscous fluid flow in
terms of its vorticity and its rate of expansion (called expansion-rate
for short). The most important fluid dynamical variable for a fluid
flow is its velocity $u(x,t)$, whose two derivatives are the divergence
$\nabla\cdot u$, measuring the rate of expansion of fluid, and the
vorticity $\omega=\nabla\wedge u$. Under appropriate boundary conditions
on the velocity, which will not be studied however in the present
work, the velocity $u$ is determined uniquely by its vorticity $\omega$
and its expansion-rate $\phi$ at any instance (The Biot-Savart law).
The vorticity and expansion-rate formulation is established by allowing
the vorticity $\omega$ and the expansion-rate $\phi$ as independent
fluid dynamical variables instead of the velocity, and considering
the continuity equation and the evolution equations for $\omega$
and $\phi$ together as the motion equations. With this point of view,
the vorticity and expansion-rate formulation naturally leads to powerful
approximation models for compressible fluid flows with negligible
fluid density gradient, which comes up as a surprising reward.

Recall that the random vortex method (cf. for example \citep{Chorin 1973,Cottet and Koumoutsakos 2000,Goodman1987,Long 1988,Majda and Bertozzi 2002})
for solving numerically the motion equations of two dimensional (2D)
incompressible viscous fluid flows is implemented based on the following
observation. The transition probability density function $p(s,x;t,y)$
(where $t>s\geq0$, $x,y\in\mathbb{R}^{2}$) of the Brownian fluid
particles with velocity $u(x,t)$ is the fundamental solution associated
with the (backward) parabolic operator $L_{u}+\partial_{t}$, where
$L_{u}=\nu\Delta+u\cdot\nabla$, $\partial_{t}$ denotes the time
derivative $\frac{\partial}{\partial t}$, and $\nu>0$ is the viscosity
constant, cf. \citep{Stroock and Varadhan 1979}. For an incompressible
fluid flow, its velocity $u(x,t)$ is solenoidal, hence the adjoint
of $L_{u}$ is $L_{-u}$ and $p(s,x;t,y)$ is the fundamental solution
of the (forward) parabolic operator $L_{-u}-\partial_{t}$ (cf. \citep[Theorem 15, page 28]{Friedman 1964}).
This fact allows us to represent the vorticity $\omega$ in terms
of the initial vorticity $\omega_{0}$ and the distribution of the
Brownian particles. Together with the Biot-Savart law, we may represent
the velocity $u(x,t)$ in terms of McKean-Vlasov type stochastic differential
equations (SDEs). The remarkable feature of this random vortex approach
is that the equations obtained which determine the velocity field
$u(x,t)$ are ordinary SDEs involving the law of their solutions.
This approach relies on two facts. Firstly for a solenoidal vector
field $u$, the adjoint operator of $L_{u}+\partial_{t}$ is $L_{-u}-\partial_{t}$
on the space-time $\mathbb{R}^{2}\times[0,\infty)$ equipped with
the product Lebesgue measure. Secondly the vorticity $\omega$ for
a 2D flow (without external force applying on the fluid) is a solution
to the parabolic equation $(L_{-u}-\partial_{t})\omega=0$ so that
the vorticity $\omega(x,t)$ may be written as an integral of the
initial vorticity against the fundamental solution $p(s,x;t,y)$.
For a three dimensional (3D) incompressible viscous fluid flow with
velocity $u$, the vorticity $\omega$ evolves according to a more
complicated vorticity equation 
\[
(L_{-u}-\partial_{t})\omega+S\omega+G=0,
\]
where $S$ is the stress tensor, a time-dependent symmetric tensor.
It is easy towork out, by using It\^o's formula, a stochastic representation
(Feynman-Kac type formula) for the vorticity in terms of the backward
flow of the diffusion with velocity $u(x,t)$, which may be used to
implement numerical schemes via backward type SDEs. While these schemes
are substantially challenging and expansive. In a recent work \citep{QSZ2022}
the authors are able to reformulate the Feynman-Kac representation
for the vorticity $\omega$ in terms of McKeak-Vlasov type ordinary
SDEs. The key proposition in \citep{QSZ2022} is to derive a path
space version of the fact that the adjoint of $L_{u}+\partial_{t}$
is $L_{-u}-\partial_{t}$ in terms of conditional distribution duality
among diffusion processes, which can be described as the following.
Let $\mathbb{P}_{u}^{\xi,\tau}$ (where $\tau\geq0$ and $\xi\in\mathbb{R}^{d}$)
be the law of the diffusion with its infinitesimal generator $L_{u}$
started from $\xi$ at time $\tau$ and $\mathbb{P}_{u}^{\xi,0\rightarrow\eta,T}$
be the pinned measure of the $L_{u}$-diffusion started at $\xi$
and ended at $\eta$ at time $T>0$. The main tool developed in \citep{QSZ2022}
is the following conditional law duality: if $u$ is solenoidal then
\begin{equation}
\mathbb{P}_{-u^{T}}^{\eta,0\rightarrow\xi,T}\circ\tau_{T}=\mathbb{P}_{u}^{\xi,0\rightarrow\eta,T}\label{Duality-01}
\end{equation}
for every $T>0$, where $u^{T}(x,t)=u(x,(T-t)^{+})$ and $\tau_{T}:C([0,T];\mathbb{R}^{d})\rightarrow C([0,T];\mathbb{R}^{d})$
is the time reversal operation sending a path $x(t)$ to $x(T-t)$.
In this paper we are going to establish a new conditional law duality
among general diffusion processes where $u$ is not necessary being
divergence-free. This generalised conditional law duality will be
a key ingredient for the present study of viscous fluid flows.

The paper is organised as the following. In Section 2, we recall the
fundamental evolution equations that the vorticity and the expansion
rate for a viscous fluid flow must satisfy, which allows us to establish
an approximation model for a compressible flow with a slowly varied
(over the space) fluid density. Our model reduces to the exact motion
equations for incompressible fluid flows, including the Oberbeck-Boussinesq
flows. A random vortex and expansion-rate method via the vorticity
and expansion rate formulation for some important compressible fluid
flows, which are discussed in Section 3, by using the powerful stochastic
integral representations to be established in the remaining part of
the paper. We establish the main mathematical tools needed in the
previous sections in the last two sections. In Section 4 we establish
the probabilistic tools in order to implement the random vortex method
for viscous compressible fluid flows. In this section we establish
the duality of conditional laws among diffusions, Theorem \ref{thm2.2},
which is one of the main contributions of the paper. By using this
important tool, we establish a new functional integral representation
theorem for a class of linear parabolic equations, Theorem \ref{thm3.1},
in Section 5. These representations are formulated in terms of forward
SDEs rather than backward stochastic flows, so our representations
are different from the classical Feynman-Kac's formulas which are
usually established with backward SDEs, cf. Freidlin \citep{Freidlin1985}. 

The paper is written in a way so that Section 4 and Section 5, in
which the new probabilistic tools are established, can be read independently
by those who are not so much interested in fluid dynamics, but wish
to learn the techniques developed there which are likely useful in
dealing with other linear and non-linear evolution equations.

\section{The motion equations of viscous fluid flows}

We shall deal with a viscous fluid flow in $\mathbb{R}^{d}$ ($d=3$
or $d=2$) without space constraint. The equations of motion for a
(compressible) viscous fluid flow may be formulated in terms of the
flow velocity $u(x,t)$ (with its components $u^{i}(x,t)$ where $i=1,\ldots,d$),
the fluid density $\rho(x,t)$, and the pressure $p(x,t)$. To simplify
our notations we make use of the following convention for tensors:
$u_{;j}^{i}$ denotes the partial derivatives $\frac{\partial}{\partial x^{j}}u^{i}$,
$S_{j}^{i}=\frac{1}{2}\left(u_{;j}^{i}+u_{;i}^{j}\right)$, and $C_{j}^{i}=S_{j}^{i}-\frac{1}{d}\phi\delta_{j}^{i}$
where $\phi=\nabla\cdot u$ is the rate of expansion.

We shall reformulate the motion equations via the vorticity $\omega=\nabla\wedge u$
(whose components $\omega^{i}=\varepsilon^{ijk}u_{;j}^{k}$ and if
$d=2$, $\omega$ is identified with $\omega^{3}=\varepsilon^{3ij}u_{;i}^{j}$)
and the expansion rate $\phi=\nabla\cdot u$, which allows us to propose
approximation equations for a viscous fluid flow with a slightly varied
fluid density. To the best knowledge of the present authors, this
approach seems different from the existing methods in literature.

\subsection{The vorticity and expansion-rate formulation}

In this part we work out motion equations which are valid for fluid
flows in terms of the vorticity $\omega$ and the rate of expansion
$\phi$. The viewpoint in this formulation is to substitute the velocity
$u$ by its differentials the vorticity $\omega=\nabla\wedge u$ and
its divergence $\phi=\nabla\cdot u$. The formulation is based on
the identity that $\Delta u=-\nabla\wedge\omega+\nabla\phi$. If $u$,
$\omega$ and $\phi$ together with their derivatives vanish at infinity,
then the velocity $u$ may be determined by $\phi$ and $\omega$
via the Biot-Savart law. Indeed by Green's formula 
\begin{equation}
u(x,\cdot)=\int_{\mathbb{R}^{d}}K_{d}(y,x)\wedge\omega(y)\textrm{d}y-\int_{\mathbb{R}^{d}}\phi(y)K_{d}(y,x)\textrm{d}y\label{BS-e1}
\end{equation}
for $x\in\mathbb{R}^{d}$, where the time parameter is suppressed,
$K_{d}(y,x)=\nabla_{y}G(y,x)$ is the Biot-Savart singular integral
kernel, and $G(y,x)$ is the Green function in $\mathbb{R}^{d}$.
Hence 
\begin{equation}
K_{d}(y,x)=\frac{1}{2(d-1)\pi}\frac{y-x}{|y-x|^{d}}\quad\textrm{ for }y\neq x\label{B-S-K1}
\end{equation}
where $d=2$ or $3$.

The equations of motion for fluid flows (viscous or inviscid, compressible
or not) are known (cf. Landau and Lifschitz \citep[Chapter II]{Landau-LifshitzFluids}):

\begin{equation}
\partial_{t}\rho+\nabla\cdot(\rho u)=0\label{C-e1}
\end{equation}
and 
\begin{equation}
\rho\left(\partial_{t}+u\cdot\nabla\right)u^{i}=\frac{\partial}{\partial x^{j}}\left(-p\delta_{j}^{i}+2\mu C_{j}^{i}+\zeta\phi\delta_{j}^{i}\right)+\rho F^{i}\label{eq:NS-g2}
\end{equation}
for $i=1,\ldots,d$, where $F=(F^{i})$ is an external force applying
to the fluid, $\mu$ and $\zeta$ are the first and second viscosity
constants respectively. Here and thereafter throughout the paper,
Einstein's convention that repeated indices are summed over their
range is used, unless otherwise specified.

The system (\ref{C-e1}, \ref{eq:NS-g2}), as a system of PDEs, contains
still more unknown dynamical variables than the number of equations,
and therefore the system (\ref{C-e1}, \ref{eq:NS-g2}) has to be
supplied with a state equation which relates the pressure in terms
of the fluid density and/or other statistical quantities such as the
temperature. A state equation will be brought in if it is needed.

For simplicity we assume that the viscosity coefficients $\mu$, $\zeta$
are constant so that the momentum equation (\ref{eq:NS-g2}) can be
rewritten

\begin{equation}
\rho\left(\partial_{t}+u\cdot\nabla\right)u=\mu\Delta u-\nabla(p-\lambda\phi)+\rho F,\label{NS-C1}
\end{equation}
where $\lambda=\zeta+\frac{d-2}{d}\mu$.

Let us derive the vorticity equation and the evolution equation of
the rate of expansion which we believe are well known in literature,
though not in the form we are going to formulate, cf. \citep{Lions1998}.

\vskip0.3cm

\emph{1) Vorticity and expansion-rate equations -- 3D case}

\vskip0.3cm

Let us work out these equations for 3D flows first. Under our convention
that $u_{;j}^{i}=\frac{\partial}{\partial x^{j}}u^{i}$ and $\rho_{;j}=\frac{\partial}{\partial x^{j}}\rho$,
by differentiating (\ref{NS-C1}) we obtain that 
\begin{align}
\rho\left(\partial_{t}+u\cdot\nabla\right)u_{;k}^{i} & =\mu\Delta u_{;k}^{i}-\rho u_{;k}^{l}u_{;l}^{i}-\rho_{;k}\left(\partial_{t}+u\cdot\nabla\right)u^{i}\nonumber \\
 & -\frac{\partial^{2}}{\partial x^{k}\partial x^{i}}(p-\lambda\phi)+(\rho F^{i})_{;k}.\label{PDE-td1}
\end{align}
By taking trace both sides of (\ref{PDE-td1}) we derive the expansion-rate
equation 
\begin{align}
\partial_{t}\phi & =\sum_{i,j}\frac{\partial}{\partial x^{j}}\left(\frac{\mu+\lambda}{\rho}\delta_{ij}\frac{\partial}{\partial x^{i}}\phi\right)-(u\cdot\nabla)\phi-\nabla\rho^{-1}\cdot\nabla p\nonumber \\
 & -\mu\nabla\rho^{-1}\cdot(\nabla\wedge\omega)+|\omega|^{2}-|\nabla u|^{2}-\frac{1}{\rho}\Delta p+\nabla\cdot F.\label{phi-1}
\end{align}
Similarly, multiplying $\varepsilon^{jki}$ both sides of (\ref{PDE-td1})
and summing up the indices $k$ and $i$ we deduce the evolution equation
of the vorticity 
\begin{align}
\partial_{t}\omega & =\sum_{i,j}\frac{\partial}{\partial x^{j}}\left(\frac{\mu}{\rho}\delta_{ij}\frac{\partial}{\partial x^{i}}\omega\right)-(u\cdot\nabla)\omega+(\omega\cdot\nabla)u-\phi\omega\nonumber \\
 & -\mu\frac{\partial\rho^{-1}}{\partial x^{j}}\nabla\omega^{j}-\nabla\rho^{-1}\wedge\nabla p+(\lambda+\mu)\nabla\rho^{-1}\wedge\nabla\phi+\nabla\wedge F.\label{om-1}
\end{align}
These equations should be understood on the region where $\rho(x,t)$
is strictly positive. These are the fundamental equations for the
study of viscous fluid flows (cf. \citep{Landau-LifshitzFluids,Lions1998}
for example). 
\begin{rem}
Here we expand out the sum for the diffusion part (writing in its
divergence form) rather than implementing Einstein's convention or
writing $\nabla\cdot\left(\frac{\mu+\lambda}{\rho}\nabla\phi\right)$,
for clarity.
\end{rem}

\begin{rem}
If the fluid is incompressible, that is, the density $\rho$ of fluid
is constant, then the continuity equation (\ref{C-e1}) is reduced
to that $\phi=\nabla\cdot u=0$, hence the equation (\ref{phi-1})
for the rate of expansion becomes the Possion equation for the pressure
$p$. More precisely, for an incompressible fluid flow 
\begin{equation}
\frac{1}{\rho}\Delta p+|\nabla u|^{2}-|\omega|^{2}-\nabla\cdot F=0.\label{inc-01}
\end{equation}
The vorticity equation is reduced to the well-known vorticity transport
equation for an incompressible fluid flow: 
\begin{equation}
\left(\sum_{i,j}\frac{\partial}{\partial x^{j}}\frac{\mu}{\rho}\delta_{ij}\frac{\partial}{\partial x^{i}}-u\cdot\nabla-\partial_{t}\right)\omega+(\omega\cdot\nabla)u-\phi\omega+\nabla\wedge F=0.\label{in-vort-a1}
\end{equation}
\end{rem}

We next write down the evolution equations for the vorticity $\omega$
and rate of expansion $\phi$ under various assumptions on the fluid
flows.

\vskip0.3cm

\emph{2) Inviscid fluid flows}

\vskip0.3cm

By inviscid fluid flows, we mean the fluid flows with vanishing viscosity
coefficients, i.e. $\mu=\zeta=\lambda=0$. The momentum equation,
the vorticity equation and the expansion-rate equation can be simplified
greatly as the following: 
\begin{equation}
\rho(\partial_{t}+u\cdot\nabla)u=-\nabla p+\rho F,\label{E-e1}
\end{equation}
which is the Euler equation. The vorticity equation for inviscid fluid
flows is simplified as the following 
\begin{equation}
\partial_{t}\omega=-(u\cdot\nabla)\omega-\phi\omega+(\omega\cdot\nabla)u-\nabla\rho^{-1}\wedge\nabla p+\nabla\wedge F.\label{inviscid2}
\end{equation}
For 2D inviscid fluid flows there is a further simplification as the
term $(\omega\cdot\nabla)u$ vanishes.

Similarly the expansion rate equation is written as 
\begin{equation}
\partial_{t}\phi=-(u\cdot\nabla)\phi-\nabla\rho^{-1}\cdot\nabla p+|\omega|^{2}-|\nabla u|^{2}-\rho^{-1}\Delta p+\nabla\cdot F.\label{inviscid3}
\end{equation}

\vskip0.3cm

\emph{3) Vorticity and expansion-rate equations -- 2D case}

\vskip0.3cm

Let us employ the following conventions. For two dimensional vectors
$a=(a^{1},a^{2})$ and $b=(b^{1},b^{2})$, $a\wedge b$ is the scalar
\[
a\wedge b=\varepsilon^{3jk}a^{j}b^{k}=a^{1}b^{2}-a^{2}b^{1}.
\]
If $c$ is a constant identifying with $(0,0,c)$ then 
\[
a\wedge c=(\varepsilon^{ij3}a^{j}c)=(a^{2}c,-a^{1}c).
\]
The vorticity $\omega$ for two dimensional flows is defined as the
following: $\nabla\wedge u=(0,0,\omega^{3})$, and
\[
\nabla\wedge\omega=\left(\frac{\partial}{\partial x^{2}}\omega,-\frac{\partial}{\partial x^{1}}\omega\right).
\]
Therefore both terms $\frac{\mu}{\rho^{2}}\frac{\partial\rho}{\partial x^{j}}\nabla\omega^{j}$
and $(\omega\cdot\nabla)u$ vanish identically. Also for 2D flows,
$\lambda=\zeta$, so that the basic equations for 2D flows are simplified
as the following: 
\begin{align}
\partial_{t}\phi & =\sum_{i,j}\frac{\partial}{\partial x^{j}}\left(\frac{\mu+\zeta}{\rho}\delta_{ij}\frac{\partial}{\partial x^{i}}\phi\right)-(u\cdot\nabla)\phi-\nabla\rho^{-1}\cdot(p+\mu\nabla\wedge\omega)\nonumber \\
 & +|\omega|^{2}-|\nabla u|^{2}-\frac{1}{\rho}\Delta p+\nabla\cdot F\label{phi-2D}
\end{align}
and 
\begin{align}
\partial_{t}\omega & =\sum_{i,j}\frac{\partial}{\partial x^{j}}\left(\frac{\mu}{\rho}\delta_{ij}\frac{\partial\omega}{\partial x^{i}}\right)-(u\cdot\nabla)\omega-\phi\omega\nonumber \\
 & +(\zeta+\mu)\nabla\rho^{-1}\wedge\nabla\phi-\nabla\rho^{-1}\wedge\nabla p+\nabla\wedge F.\label{om-2D}
\end{align}

\vskip0.3truecm

\emph{4) Polytropic gas type flows}

\vskip0.3truecm

The equations of motion in general have to be supplemented by a state
equation. For polytropic gases (cf. \citep{ChenFeldman2018} for example),
this is achieved by the state equation: $p=\kappa\rho^{\gamma}$ where
$\gamma>1$ and $\kappa>0$ are two constants. In particular, $\nabla\rho\wedge\nabla p=0$.

In general we say a fluid flow is polytropic gas type if the pressure
$p$ is a function of its fluid density $\rho$ only. Thus for a polytropic
type gas flow, the term $\nabla\rho\wedge\nabla p$ appearing in the
vorticity equation vanishes identically. In particular, for a two
dimensional polytropic gas type fluid flow, the rate of expansion
equation and the vorticity equation are given as the following: 
\begin{align}
\partial_{t}\phi & =\sum_{i,j}\frac{\partial}{\partial x^{j}}\left(\frac{\mu+\zeta}{\rho}\delta_{ij}\frac{\partial}{\partial x^{i}}\phi\right)-(u\cdot\nabla)\phi\nonumber \\
 & -\nabla\rho^{-1}\cdot(\mu\nabla\wedge\omega+\nabla p)+|\omega|^{2}-|\nabla u|^{2}-\frac{1}{\rho}\Delta p+\nabla\cdot F\label{phi-poly2D-1}
\end{align}
and 
\begin{align}
\partial_{t}\omega & =\sum_{i,j}\frac{\partial}{\partial x^{j}}\left(\frac{\mu}{\rho}\delta_{ij}\frac{\partial}{\partial x^{i}}\omega\right)-(u\cdot\nabla)\omega-\phi\omega+\nabla\wedge F\nonumber \\
 & +(\zeta+\mu)\nabla\rho^{-1}\wedge\nabla\phi.\label{om-2D-1}
\end{align}

\subsection{Approximation equations for compressible flows}

As a consequence of the vorticity and expansion rate formulation for
fluid flows, we are now in a position to uncover an approximation
system for a compressible fluid flow with a negligible fluid density
gradient. Suppose the viscosity constants are small, and the fluid
is nearly incompressible so that, the fluid density varies little
over the space, and therefore the gradient of the density is negligible.
Then 
\[
-\nabla\rho^{-1}\cdot\nabla p\approx0,\quad-\mu\nabla\rho^{-1}\cdot(\nabla\wedge\omega)\approx0
\]
and
\[
-\nabla\rho^{-1}\wedge\nabla p\approx0,\quad(\lambda+\mu)\nabla\rho^{-1}\wedge\nabla\phi\approx0,\quad-\mu\frac{\partial\rho^{-1}}{\partial x^{j}}\nabla\omega^{j}\approx0.
\]
Hence these terms may be ignored in numerical schemes. Also note that
if the fluid is incompressible, then 
\[
|\omega|^{2}-|\nabla u|^{2}-\frac{1}{\rho}\Delta p+\nabla\cdot F
\]
vanishes identically, thus for nearly compressible flows, this term
may be ignored in the expansion-rate equation.

Therefore we are led to the following approximation equations for
a compressible fluid flow with a negligible fluid density gradient:
\begin{equation}
\partial_{t}\rho+(u\cdot\nabla)\rho+\phi\rho=0,\label{AE-c1}
\end{equation}
\begin{equation}
(\mu+\lambda)\sum_{i,j}\frac{\partial}{\partial x^{j}}\left(\frac{1}{\rho}\delta_{ij}\frac{\partial}{\partial x^{i}}\phi\right)-(u\cdot\nabla)\phi-\partial_{t}\phi=0,\label{AE-c2}
\end{equation}
\begin{equation}
\mu\sum_{i,j}\frac{\partial}{\partial x^{j}}\left(\frac{1}{\rho}\delta_{ij}\frac{\partial}{\partial x^{i}}\omega\right)-(u\cdot\nabla)\omega-\partial_{t}\omega+(\omega\cdot\nabla)u-\phi\omega+\nabla\wedge F=0\label{AE-c3}
\end{equation}
together with the definition that 
\begin{equation}
\nabla\cdot u=\phi\quad\textrm{ and }\quad\nabla\wedge u=\omega.\label{AE-r1}
\end{equation}

The previous system of approximation motion equations becomes exact
for incompressible fluid flows. Note that a state equation can be
avoided in the previous system.

The relations in (\ref{AE-r1}) show that the velocity $u(x,t)$ (with
appropriate boundary conditions at infinity) is determined by the
Biot-Savart law (\ref{BS-e1}), which in turn implies that $u$ is
a non-local function of the expansion rate $\phi$ and the vorticity
$\omega$. Therefore $u$ does not count as an independent dynamical
variable in equations (\ref{AE-c1}, \ref{AE-c2}, \ref{AE-c3}).
The advantage for taking (\ref{AE-c1}, \ref{AE-c2}, \ref{AE-c3},
\ref{AE-r1}) as approximation for compressible fluid flows lies in
the fact that the system does not contain explicitly the pressure
$p$, and if we count $\rho$, $\phi$ and $\omega$ as independent
fluid dynamical variables rather than $\rho,u$ and the pressure $p$,
then there are exactly five independent dynamical variables and five
equations (\ref{AE-c1}, \ref{AE-c2}, \ref{AE-c3}). In particular,
for the approximation motion equations of the compressible fluid flows,
a state equation is no longer needed for closing the fluid dynamics,
which is in fact the case for incompressible flows. 
\begin{rem}
For a two dimensional polytropic type gas fluid flow, the approximation
vorticity equation (\ref{AE-c3}) differers from the exact one only
by a term $(\zeta+\mu)\nabla\rho^{-1}\wedge\nabla\phi$ which should
be rather small if the fluid flow is not so strongly compressible. 
\end{rem}

We end up this part by a short discussion of inviscid fluid flows.
For an inviscid fluid flow with a negligible fluid density gradient,
then the approximation motion equations are given as the following:
\begin{equation}
\partial_{t}\rho+(u\cdot\nabla)\rho+\phi\rho=0,\label{AE-c1-2}
\end{equation}
\begin{equation}
\partial_{t}\phi+(u\cdot\nabla)\phi=0,\label{AE-c2-2}
\end{equation}
\begin{equation}
\partial_{t}\omega+(u\cdot\nabla)\omega-(\omega\cdot\nabla)u+\phi\omega-\nabla\wedge F=0\label{AE-c3-2}
\end{equation}
and 
\begin{equation}
\nabla\cdot u=\phi\quad\textrm{ and }\quad\nabla\wedge u=\omega.\label{AE-r1-2}
\end{equation}

\section{Random vortex and expansion-rate formulation}

In this section we establish the functional integral representation
for solutions of the approximation equations (\ref{AE-c1}, \ref{AE-c2},
\ref{AE-c3}, \ref{AE-r1}) for a compressible with a negligible fluid
density gradient, in terms of ordinary McKean-Vlasov stochastic differential
equations. In addition to the notations established in the previous
section, we shall introduce several notions and notations for stating
our results, which will be discussed further in Section 4 and Section
5. The new stochastic representations for solutions of linear parabolic
equations in terms of SDEs shall be established in Sections 4 and
5. 

Given a time-dependent vector field $b(x,t)$ on $\mathbb{R}^{d}$
(where $d=3$ or $d=2$ in this section), we introduce a family of
elliptic operators
\[
L_{b;\kappa}=\kappa\sum_{i,j}\frac{\partial}{\partial x^{j}}\frac{1}{\rho}\delta_{ij}\frac{\partial}{\partial x^{i}}\omega+b\cdot\nabla
\]
where $\kappa>0$. Let $q_{j}^{i}=S_{j}^{i}-\phi\delta_{ij}$. Then
the approximation equations (\ref{AE-c2}, \ref{AE-c3}) may be written
as 
\begin{equation}
\left(L_{-u;\lambda+\mu}-\partial_{t}\right)\phi=0\label{AE-v1}
\end{equation}
and 
\begin{equation}
\left(L_{-u;\mu}-\partial_{t}\right)\omega^{i}+q_{j}^{i}\omega^{j}+f^{i}=0\label{AE-v2}
\end{equation}
where $f=\nabla\wedge F$, i.e. $f^{i}=\varepsilon^{ijk}\frac{\partial}{\partial x^{j}}F^{k}$.
Recall that $\phi$ is the expansion-rate $\phi$ and $\omega$ is
the vorticity, and $u$ is determined by the equations: $\nabla\cdot u=\phi$
and $\nabla\wedge u=\omega$.

Let $\phi_{0}=\phi(\cdot,0)$ and $\omega_{0}=\omega(\cdot,0)$ be
the initial rate of expansion and the initial vorticity. 

$\varOmega=C([0,\infty);\mathbb{R}^{d})$ denotes the space of continuous
paths in $\mathbb{R}^{d}$ which is the sample space, and $X=(X_{t})_{t\geq0}$
is the coordinate process on $\varOmega$. That is, for each $t\geq0$,
$X_{t}:\varOmega\mapsto\mathbb{R}^{d}$ defined by $X_{t}(\psi)=\psi(t)$
for every $\psi\in\varOmega$. Let $\mathcal{F}=\sigma\left\{ X_{t}:t\geq0\right\} $
be the smallest $\sigma$-algebra on $\varOmega$ so that $X_{t}$
are measurable for all $t\geq0$. For each $\eta\in\mathbb{R}^{d}$,
$\mathbb{P}_{b;\kappa}^{\eta}$ denotes the distribution of the $L_{b;\kappa}$-diffusion
starting from $\eta\in\mathbb{R}^{d}$ at time $0$, which is the
unique probability measure on $(\varOmega,\mathcal{F})$ such that
\[
\mathbb{P}_{b;\kappa}^{\eta}\left[X_{0}=\eta\right]=1
\]
and
\[
M_{t}^{[h]}=h(X_{t},t)-h(\eta,0)-\int_{0}^{t}L_{b;\kappa}h(X_{s},s)\textrm{d}s
\]
(for $t\geq0$) is a continuous local martingale on $(\varOmega,\mathcal{F},\mathbb{P}_{b;\kappa}^{\eta})$
for every $h\in C^{2,1}(\mathbb{R}^{d}\times[0,\infty))$, cf. Stroock
and Varadhan \citep{Stroock and Varadhan 1979}. $p_{b;\kappa}(s,x;t,y)$
(for $t>s\geq0$, $x$ and $y$ belong to $\mathbb{R}^{d}$) denotes
the transition probability density function of the $L_{b;\kappa}$-diffusion.

If $\eta,\xi\in\mathbb{R}^{d}$ and $T>0$, then $\mathbb{P}_{b;\kappa}^{\eta,0\rightarrow\xi,T}$
denotes the conditional law of the $L_{b;\kappa}$-diffusion $X$
starting from $\eta$ at time zero given that $X_{T}=\xi$. Formally
\[
\mathbb{P}_{b;\kappa}^{\eta,0\rightarrow\xi,T}\left[\cdot\right]=\mathbb{P}_{b;\kappa}^{\eta}\left[\left.\cdot\right|X_{T}=\xi\right],
\]
for the precise definition see Section 4 below. 

Let $K_{d}(y,x)$ denote the Biot-Savart singular integral kernel
on $\mathbb{R}^{d}$ given in (\ref{B-S-K1}).
\begin{thm}
\label{thm5.1-1}Let $u(x,t)$ be the solution of the approximate
compressible fluid flow equations (\ref{AE-c1}, \ref{AE-c2}, \ref{AE-c3})
and (\ref{AE-r1}). For every path $\psi\in\varOmega$ and $T>0$,
$t\mapsto R_{j}^{i}(\psi,T;t)$ are the unique solutions to the ordinary
differential equations 
\begin{equation}
dR_{j}^{i}(\psi,T;t)=-R_{k}^{i}(\psi,T;t)S_{j}^{k}(\psi(t),t),\quad R_{j}^{i}(\psi,T;T)=\delta_{ij}.\label{R-g1}
\end{equation}
Then 
\begin{align}
u(x,T) & =\int_{\mathbb{R}^{3}}\mathbb{P}_{u;\mu}^{\eta}\left[K_{3}(X_{T},x)\wedge R(X,T;0)\omega_{0}(\eta)\right]\textrm{d}\eta\nonumber \\
 & +\int_{0}^{T}\int_{\mathbb{R}^{3}}\mathbb{P}_{u;\mu}^{\eta}\left[\textrm{e}^{\int_{0}^{t}\phi(X_{s},s)\textrm{d}s}K_{3}(X_{T},x)\wedge R(X,T;t)f(X_{t},t)\right]\textrm{d}\eta\textrm{d}t\nonumber \\
 & -\int_{\mathbb{R}^{3}}\mathbb{P}_{u;\lambda+\mu}^{\eta}\left[\textrm{e}^{\int_{0}^{T}\phi(X_{s},s)\textrm{d}s}K_{3}(X_{T},x)\right]\phi_{0}(\eta)\textrm{d}\eta\label{AC2-rep2}
\end{align}
for every $x\in\mathbb{R}^{3}$ and $T>0$. 
\end{thm}

\begin{proof}
Let $T>0$ be arbitrary but fixed. By applying the stochastic representation
Theorem \ref{thm3.1} to $\phi$ (with $a^{ij}=(\lambda+\mu)\frac{1}{\rho}\delta_{ij}$
and $b=-u$) we obtain that 
\begin{equation}
\phi(\xi,T)=\int_{\mathbb{R}^{3}}\mathbb{P}_{u;\lambda+\mu}^{\eta,0\rightarrow\xi,T}\left[\textrm{e}^{\int_{0}^{T}\phi(X_{s},s)\textrm{d}s}\phi_{0}(\eta)\right]p_{u;\lambda+\mu}(0,\eta;T,\xi)\textrm{d}\eta\label{com-rep1}
\end{equation}
for $\xi\in\mathbb{R}^{3}$. To work out an implicit representation
for the vorticity $\omega$, one applies Theorem \ref{thm3.1} to
(\ref{AE-v2}) too. Therefore we define a gauge functional $t\mapsto A(\psi,T;t)$
for every continuous path $\psi$ and $T>0$ by solving the system
of ODEs: 
\begin{equation}
\frac{\textrm{d}}{\textrm{d}t}A_{j}^{i}(\psi,T;t)=-A_{k}^{i}(\psi,T;t)S_{j}^{k}(\psi(t),t)+A_{j}^{i}(\psi,T;t)\phi(\psi(t),t)\label{com-rep2}
\end{equation}
and 
\begin{equation}
A_{j}^{i}(\psi,T;T)=\delta_{ij}.\label{com-rep3}
\end{equation}
Then, by using Theorem \ref{thm3.1} to the vorticity equation (\ref{AE-v2})
we obtain the following functional integral representation
\begin{align}
\omega^{i}(\xi,T) & =\int_{\mathbb{R}^{3}}\mathbb{P}_{u;\mu}^{\eta,0\rightarrow\xi,T}\left[\textrm{e}^{\int_{0}^{T}\phi(X_{s},s)\textrm{d}s}A_{j}^{i}(X,T;0)\omega_{0}^{j}(\eta)\right]p_{u;\mu}(0,\eta;T,\xi)\textrm{d}\eta\nonumber \\
 & +\int_{0}^{T}\int_{\mathbb{R}^{3}}\mathbb{P}_{u;\mu}^{\eta,0\rightarrow\xi,T}\left[\textrm{e}^{\int_{0}^{T}\phi(X_{s},s)\textrm{d}s}A_{j}^{i}(X,T;s)f^{j}(X_{s},s)\right]p_{u;\mu}(0,\eta;T,\xi)\textrm{d}\eta\textrm{d}s.\label{com-rep4}
\end{align}
While $R_{j}^{i}(\psi,T;t)=\textrm{e}^{\int_{t}^{T}\phi(\psi(s),s)\textrm{d}s}A_{j}^{i}(\psi,T;t)$
(for $t\in[0,T]$) are the unique solutions to (\ref{R-g1}). Thus
(\ref{com-rep4}) can be rewritten as 
\begin{align}
\omega^{i}(\xi,T) & =\int_{\mathbb{R}^{3}}\mathbb{P}_{u;\mu}^{\eta,0\rightarrow\xi,T}\left[R_{j}^{i}(X,T;0)\omega_{0}^{j}(\eta)\right]p_{u;\mu}(0,\eta;T,\xi)\textrm{d}\eta\nonumber \\
 & +\int_{0}^{T}\int_{\mathbb{R}^{3}}\mathbb{P}_{u;\mu}^{\eta,0\rightarrow\xi,T}\left[\textrm{e}^{\int_{0}^{t}\phi(X_{s},s)\textrm{d}s}R_{j}^{i}(X,T;t)f^{j}(X_{t},t)\right]p_{u;\mu}(0,\eta;T,\xi)\textrm{d}\eta\textrm{d}t.\label{com-rep4-1}
\end{align}
Finally we apply the Biot-Savart law (\ref{BS-e1}) and write $u(x,T)$
in terms of the singular integrals against $\phi$ and $\omega$.
More precisely, since 
\[
u(x,T)=\int_{\mathbb{R}^{3}}K_{3}(\xi,x)\wedge\omega(\xi,T)\textrm{d}\xi-\int_{\mathbb{R}^{3}}K_{3}(\xi,x)\phi(\xi,T)\textrm{d}\xi.
\]
Substituting $\omega(\xi,T)$ and $\phi(\xi,T)$ in the previous equation
via (\ref{com-rep1}) and (\ref{com-rep4-1}) respectively, and using
the Fubini theorem, one obtains
\begin{align*}
u(x,T) & =\int_{\mathbb{R}^{3}}\int_{\mathbb{R}^{3}}\mathbb{P}_{u;\mu}^{\eta,0\rightarrow\xi,T}\left[K_{3}(X_{T},x)\wedge R(X,T;0)\omega_{0}(\eta)\right]p_{u;\mu}(0,\eta;T,\xi)\textrm{d}\xi\textrm{d}\eta\\
 & +\int_{0}^{T}\int_{\mathbb{R}^{3}}\int_{\mathbb{R}^{3}}\mathbb{P}_{u;\mu}^{\eta,0\rightarrow\xi,T}\left[\textrm{e}^{\int_{0}^{t}\phi(X_{s},s)\textrm{d}s}K_{3}(X_{T},x)\wedge R(X,T;t)f(X_{t},t)\right]p_{u;\mu}(0,\eta;T,\xi)\textrm{d}\xi\textrm{d}\eta\textrm{d}t\\
 & -\int_{\mathbb{R}^{3}}\int_{\mathbb{R}^{3}}\mathbb{P}_{u;\lambda+\mu}^{\eta,0\rightarrow\xi,T}\left[\textrm{e}^{\int_{0}^{T}\phi(X_{s},s)\textrm{d}s}\phi_{0}(\eta)K_{3}(X_{T},x)\right]p_{u;\lambda+\mu}(0,\eta;T,\xi)\textrm{d}\xi\textrm{d}\eta.
\end{align*}
The representation then follows by integrating the variable $\xi$
and the integration formula:
\[
\int_{\mathbb{R}^{3}}\mathbb{P}_{u;\mu}^{\eta,0\rightarrow\xi,T}\left[H\right]p_{u;\mu}(0,\eta;T,\xi)\textrm{d}\xi=\mathbb{P}_{u;\mu}^{\eta,0}\left[H\right]
\]
for any integrable $H$ which $\mathcal{F}_{T}$-measurable, where
$\mathcal{F}_{T}=\sigma\left\{ X_{t}:t\leq T\right\} $, which completes
the proof.
\end{proof}
There is a significant simplification for two dimensional flows, which
is useful in numerical schemes for solving the approximate dynamical
equations of compressible fluid flows. 
\begin{thm}
\label{thm5.1}Suppose $u(x,t)$ be the solution to the approximation
2D compressible fluid flows: 
\begin{equation}
\partial_{t}\rho=-(u\cdot\nabla)\rho-\phi\rho,\label{AE-c1-1}
\end{equation}
\begin{equation}
\partial_{t}\phi=(\mu+\lambda)\sum_{i,j}\frac{\partial}{\partial x^{j}}\left(\frac{1}{\rho}\delta_{ij}\frac{\partial}{\partial x^{i}}\phi\right)-(u\cdot\nabla)\phi,\label{AE-c2-1}
\end{equation}
\begin{equation}
\partial_{t}\omega=\mu\sum_{i,j}\frac{\partial}{\partial x^{j}}\left(\frac{1}{\rho}\delta_{ij}\frac{\partial}{\partial x^{i}}\omega\right)-(u\cdot\nabla)\omega-\phi\omega+\nabla\wedge F\label{AE-c3-1}
\end{equation}
and 
\begin{equation}
\nabla\cdot u=\phi\quad\textrm{ and }\quad\nabla\wedge u=\omega.\label{AE-r1-1}
\end{equation}
Then 
\begin{align}
u(x,T) & =\int_{\mathbb{R}^{2}}\mathbb{P}_{u;\mu}^{\eta}\left[K_{2}(X_{T},x)\wedge\omega_{0}(\eta)\right]\textrm{d}\eta\nonumber \\
 & +\int_{0}^{T}\int_{\mathbb{R}^{2}}\mathbb{P}_{u;\mu}^{\eta}\left[\textrm{e}^{\int_{0}^{t}\phi(X_{s},s)\textrm{d}s}K_{2}(X_{T},x)\wedge f(X_{t},t)\right]\textrm{d}\eta\textrm{d}t\nonumber \\
 & -\int_{\mathbb{R}^{2}}\mathbb{P}_{u;\lambda+\mu}^{\eta}\left[\textrm{e}^{\int_{0}^{T}\phi(X_{s},s)\textrm{d}s}K_{2}(X_{T},x)\right]\phi_{0}(\eta)\textrm{d}\eta,\label{2D-AC1}
\end{align}
where $f=\frac{\partial}{\partial x_{1}}F^{2}-\frac{\partial}{\partial x_{2}}F^{1}$,
$K_{2}(y,x)$ is the Biot-Savart kernel in $\mathbb{R}^{2}$. 
\end{thm}

\begin{proof}
Indeed, for two dimensional flows, since $(\omega\cdot\nabla)u$ vanishes
identically, the gauge functional $R$ solves the following ordinary
differential equations 
\begin{equation}
\frac{\textrm{d}}{\textrm{d}t}R_{j}^{i}(\psi,T;t)=R_{j}^{i}(\psi,T;t)\phi(\psi(t),t)\label{com-rep2-1}
\end{equation}
and 
\begin{equation}
R_{j}^{i}(\psi,T;T)=\delta_{ij}.\label{com-rep3-1}
\end{equation}
which has the solution 
\[
R_{j}^{i}(\psi,T;t)=\delta_{ij}\textrm{e}^{\int_{T}^{t}\phi(\psi(s),s)\textrm{d}s}
\]
and therefore the integral representation from the previous representation
(3D case) immediately. 
\end{proof}
Based on the stochastic integral representation, Theorem \ref{thm5.1-1}
and Theorem \ref{thm5.1}, we may design numerical schemes for solving
the approximation motion equations for compressible fluid flows with
negligible fluid density gradient. We will use the so called multiple
Brownian fluid particle method, cf. \citep{LiQianXu2023}. These methods
have been studied by various researchers in the past but mainly for
incompressible flows, cf. \citep{Majda and Bertozzi 2002} for an
overview.

For simplicity we describe the random vortex and expansion-rate method
based on the vorticity and expansion-rate formulation for a 2D approximation
system of a compressible fluid flow with negligible fluid density
gradient, which is the system of partial differential equations 
\begin{equation}
\partial_{t}\rho+(u\cdot\nabla)\rho+\phi\rho=0,\label{AE-c1-3}
\end{equation}
\begin{equation}
(\mu+\lambda)\sum_{i,j}\frac{\partial}{\partial x^{j}}\left(\frac{1}{\rho}\delta_{ij}\frac{\partial}{\partial x^{i}}\phi\right)-(u\cdot\nabla)\phi-\partial_{t}\phi=0,\label{AE-c2-3}
\end{equation}
\begin{equation}
\mu\sum_{i,j}\frac{\partial}{\partial x^{j}}\left(\frac{1}{\rho}\delta_{ij}\frac{\partial}{\partial x^{i}}\omega\right)-(u\cdot\nabla)\omega-\phi\omega-\partial_{t}\omega=0\label{AE-c3-3}
\end{equation}
and 
\begin{equation}
\nabla\cdot u=\phi\quad\textrm{ and }\quad\nabla\wedge u=\omega.\label{AE-r1-3}
\end{equation}
Suppose that the initial data $u_{0}$ and the initial density $\rho_{0}$
(which is bounded, and is bounded away from zero), where $\nabla\rho_{0}$
is small. Introduce two types of Brownian particles, as in \citep{LiQianXu2023},
called $X$ and $Y$ particles, which are the diffusion processes
with infinitesimal generator $L_{u;\lambda+\mu}$ and $L_{u;\mu}$
respectively. The $X$ and $Y$ particles may be defined by solving
the stochastic differential equations: 
\begin{equation}
\textrm{d}X_{t}=\sqrt{2(\lambda+\mu)\rho^{-1}(X_{t},t)}\textrm{d}B_{t}^{1}+\left((\lambda+\mu)\nabla\rho^{-1}+u\right)(X_{t},t)\textrm{d}t,\quad X_{0}\in\mathbb{R}^{2}\label{X-particle}
\end{equation}
and 
\begin{equation}
\textrm{d}Y_{t}=\sqrt{2\mu\rho^{-1}(Y_{t},t)}\textrm{d}B_{t}^{2}+\left(\mu\nabla\rho^{-1}+u\right)(Y_{t},t)\textrm{d}t,\quad Y_{0}\in\mathbb{R}^{2}\label{Y-particle}
\end{equation}
on a probability space $(\varOmega,\mathcal{F},\mathbb{P})$, where
$B^{1}$ and $B^{2}$ are two independent Brownian motions in $\mathbb{R}^{2}$.
These diffusions may be called Brownian fluid particles with velocity
$u(x,t)$. We will use $X^{\xi}$ or $X(\xi,t)$ to denote the solution
to SDE (\ref{X-particle}) such that $X_{0}=\xi$, where $\xi\in\mathbb{R}^{2}$.
The same convention applies to $Y$-particles as well. Then, according
to Theorem \ref{thm5.1} 
\begin{align}
u(x,t) & =\int_{\mathbb{R}^{2}}\mathbb{P}\left[K_{2}(Y_{t}^{\eta},x)\wedge\omega_{0}(\eta)\right]\textrm{d}\eta\nonumber \\
 & +\int_{0}^{t}\int_{\mathbb{R}^{2}}\mathbb{P}\left[\textrm{e}^{\int_{0}^{s}\phi(Y_{r}^{\eta},r)\textrm{d}r}K_{2}(Y_{t}^{\eta},x)\wedge f(Y_{s}^{\eta},s)\right]\textrm{d}\eta\textrm{d}s\nonumber \\
 & -\int_{\mathbb{R}^{2}}\mathbb{P}\left[\textrm{e}^{\int_{0}^{t}\phi(X_{s}^{\eta},s)\textrm{d}s}K_{2}(X_{t}^{\eta},x)\right]\phi_{0}(\eta)\textrm{d}\eta,\label{2D-AC1-1}
\end{align}
together with the transport equation for $\rho^{-1}$: 
\begin{equation}
\partial_{t}\rho^{-1}+(u\cdot\nabla)\rho^{-1}-\phi\rho^{-1}=0,\label{a-eq1}
\end{equation}
we thus obtain an implicit system which allows to iterate to find
numerically the velocity $u(x,t)$.

In the previous random vortex scheme, we have to update $\nabla\rho^{-1}$
which requires to replace the singular kernel $K_{2}$ by its mollifiers.
While for flows with slowly varying fluid density, the term $\nabla\rho^{-1}$
(which is negligible), appearing together with the viscosity constants
which are also small, can be ignored in the scheme. Therefore we may
modify the $X$, $Y$ particles by solving the following SDEs: 
\begin{equation}
\textrm{d}X_{t}=\sqrt{2(\lambda+\mu)\rho^{-1}(X_{t},t)}\textrm{d}B_{t}^{1}+u(X_{t},t)\textrm{d}t,\quad X_{0}\in\mathbb{R}^{2}\label{X-particle-1}
\end{equation}
and 
\begin{equation}
\textrm{d}Y_{t}=\sqrt{2\mu\rho^{-1}(Y_{t},t)}\textrm{d}B_{t}^{2}+u(Y_{t},t)\textrm{d}t,\quad Y_{0}\in\mathbb{R}^{2}\label{Y-particle-1}
\end{equation}
instead. The modified scheme will greatly reduce the computational
cost, avoid the iteration of any derivatives of $\phi$, $\omega$
or $u$, and therefore no smoothing procedure for the singular kernel
$K_{2}$ is required. 

\section{Duality of conditional laws}

In this and the next sections, we develop the probabilistic tools
used for implementing the random vortex and expansion-rate method
for compressible flows worked out in the previous section.

Consider the following type of second order elliptic operator 
\begin{equation}
\mathscr{L}_{a;b,c}=\sum_{i,j=1}^{d}\frac{\partial}{\partial x^{j}}a^{ij}\frac{\partial}{\partial x^{i}}+\sum_{i=1}^{d}b^{i}\frac{\partial}{\partial x^{i}}+c\label{div-oper-01}
\end{equation}
in the Euclidean space $\mathbb{R}^{d}$, where $a(x,t)=(a^{ij}(x,t))_{i,j\leq d}$
(for $(x,t)\in\mathbb{R}^{d}\times[0,\infty)$) is a smooth $d\times d$
symmetric matrix-valued function with bounded derivatives which satisfies
the \emph{uniformly elliptic} condition that there is a constant $\lambda\geq1$
such that 
\begin{equation}
\lambda^{-1}|\xi|^{2}\leq\sum_{i,j=1}^{d}\xi^{i}\xi^{j}a^{ij}(x,t)\leq\lambda|\xi|^{2}\label{u-c01}
\end{equation}
for every $\xi=(\xi^{1},\cdots,\xi^{d})\in\mathbb{R}^{d}$ and $(x,t)\in\mathbb{R}^{d}\times[0,\infty)$,
$b(x,t)=(b^{1}(x,t),\ldots,b^{d}(x,t))$ is a smooth time-dependent
vector field, and $c(x,t)$ is a continuous bounded scalar function
on $\mathbb{R}^{d}\times[0,\infty)$.

It is easy to see that the adjoint $\mathscr{L}_{a;b,c}^{\star}=\mathscr{L}_{a;-b,c-\nabla\cdot b}$.

If $c=0$, then $\mathscr{L}_{a;b,0}$ is a diffusion operator, denoted
by $\mathscr{L}_{a;b}$, and its adjoint $\mathscr{L}_{a;b}^{\star}$
equals $\mathscr{L}_{a;-b,-\nabla\cdot b}$, which is however no longer
a diffusion operator unless $\nabla\cdot b=0$. In particular we can
not expect the duality (\ref{Duality-01}) for a vector field with
non-trivial divergence. We are going to establish a new duality theorem
among $\mathscr{L}_{a;b}$-diffusion pinned measures for arbitrary
time-dependent vector fields $b(x,t)$ and diffusion coefficients
$a(x,t)$. To this end we recall first several fundamental probabilistic
structures determined by the elliptic operator $\mathscr{L}_{a;b}$.

Under our assumptions on $a,b$ and $c$, there is a unique fundamental
solution $\varSigma_{a;b,c}(x,t;\xi,\tau)$ where $0\leq\tau<t\leq T$
(resp. $\varSigma_{a;b,c}^{\star}(x,t;\xi,\tau)$ for $0\leq t<\tau\leq T$)
of $\mathscr{L}_{a;b,c}-\partial_{t}$ (resp. $\mathscr{L}_{a;b,c}+\partial_{t}$).
That is 
\[
\left(\mathscr{L}_{a;b,c}-\partial_{t}\right)\varSigma_{a;b,c}(x,t;\xi,\tau)=0\;\textrm{ in }\mathbb{R}^{d}\times(\tau,T]
\]
for every pair $(\xi,\tau)\in\mathbb{R}^{d}\times[0,T]$, (resp: 
\[
\left(\mathscr{L}_{a;b,c}+\partial_{t}\right)\varSigma_{a;b,c}^{\star}(x,t;\xi,\tau)=0\;\textrm{ in }\mathbb{R}^{d}\times[0,\tau)
\]
for $(\xi,\tau)\in\mathbb{R}^{d}\times[0,T]$), and for every $\varphi\in C_{b}(\mathbb{R}^{d})$
\[
\lim_{t\downarrow\tau}\int_{\mathbb{R}^{d}}\varSigma_{a;b,c}(x,t;\xi,\tau)\varphi(\xi)\textrm{d}\xi=\varphi(x)
\]
(resp. 
\[
\lim_{t\uparrow\tau}\int_{\mathbb{R}^{d}}\varSigma_{a;b,c}^{\star}(x,t;\xi,\tau)\varphi(\xi)\textrm{d}\xi=\varphi(x)
\]
for every $x\in\mathbb{R}^{d}$).
\begin{rem}
Similarly there are unique fundamental solutions associated with $\mathscr{L}_{a;b,c}-\partial_{t}$
and $\mathscr{L}_{a;b,c}+\partial_{t}$ on a domain $D\subset\mathbb{R}^{d}$
with Lipschitz continuous boundary $\partial D$, subject to the Dirichlet
boundary condition at $\partial D$, cf. \citep{Friedman 1964}. 
\end{rem}

The following properties about the fundamental solutions are important
in our study, which can be easily verified, cf. \citep{Friedman 1964}
for details.

\vskip0.3cm

1) For every $T>0$, $\varSigma_{a;b,c}^{\star}(x,T-t;\xi,T-\tau)$
(where $0\leq\tau<t\leq T$ and $x,\xi\in\mathbb{R}^{d}$) is a fundamental
solution to $\mathscr{L}_{a^{T};b^{T},c^{T}}-\partial_{t}$, so that
\begin{equation}
\varSigma_{a;b,c}^{\star}(x,T-t;\xi,T-\tau)=\varSigma_{a^{T};b^{T},c^{T}}(x,t;\xi,\tau)\label{duality-31}
\end{equation}
for any $\xi,x\in\mathbb{R}^{d}$ and $0\leq\tau<t\leq T$, where
we have used the convention that if $f(x,t)$ is a time dependent
tensor field, then $f^{T}(x,t)=f(x,(T-t)^{+})$.

\vskip0.3cm

2) Since $\mathscr{L}_{a;b,c}^{\star}=\mathscr{L}_{a;-b,c-\nabla\cdot b}$,
the forward and backward solutions satisfy the following relation:
\begin{equation}
\varSigma_{a;-b,c-\nabla\cdot b}^{\star}(\xi,\tau;x,t)=\varSigma_{a;b,c}(x,t;\xi,\tau)\label{duality-32}
\end{equation}
for all $\xi,x\in\mathbb{R}^{d}$ and $0\leq\tau<t$, cf. \citep[Theorem 15, page 28]{Friedman 1964}.
Therefore 
\begin{equation}
\varSigma_{a;b,c}^{\star}(x,T-t;\xi,T-\tau)=\varSigma_{a^{T};-b^{T},c^{T}-\nabla\cdot b^{T}}^{\star}(\xi,\tau;x,t)\label{duality-33}
\end{equation}
for $0\leq\tau<t\leq T$ and $x,\xi\in\mathbb{R}^{d}$.

\vskip0.3cm

3) We will use the following special case of (\ref{duality-33}):
\begin{equation}
\varSigma_{a;b}^{\star}(x,T-t;\xi,T-\tau)=\varSigma_{a^{T};-b^{T},-\nabla\cdot b^{T}}^{\star}(\xi,\tau;x,t)\label{duality-34}
\end{equation}
for all $\xi,x\in\mathbb{R}^{d}$ and for $0\leq\tau<t\leq T$.

\vskip0.4cm

Let $\varOmega=C([0,\infty),\mathbb{R}^{d})$ and $X=(X_{t})_{t\geq0}$
be the coordinate process on $\varOmega$. Let $\mathcal{F}_{t}^{0}=\sigma\{X_{s}:s\leq t\}$
and $\mathcal{F}^{0}=\sigma\{X_{s}:s\geq0\}$ which coincides with
the Borel $\sigma$-algebra on $\varOmega$. For every $\xi\in\mathbb{R}^{d}$
and $\tau\geq0$ there is a unique probability measures $\mathbb{P}_{a;b}^{\xi,\tau}$
on $(\varOmega,\mathcal{F}^{0})$ such that $X=(X_{t})_{t\geq0}$
is a diffusion with generator $\mathscr{L}_{a;b}$ and 
\[
\mathbb{P}_{a;b}^{\xi,\tau}\left[X_{s}=\xi\textrm{ for all }0\leq s\leq\tau\right]=1.
\]
The distribution of $X_{t}$ (for $t>\tau$) under $\mathbb{P}_{a;b}^{\xi,\tau}$
has a density with respect to the Lebesgue measure denoted by $p_{a;b}(\tau,\xi;t,x)$.
It is a matter of fact that 
\begin{equation}
p_{a;b}(\tau,\xi;t,x)=\varSigma_{a;b}^{\star}(\xi,\tau;x,t)=\varSigma_{a;-b,-\nabla\cdot b}(x,t;\xi,\tau)\label{pdf-back-01}
\end{equation}
for all $0\leq\tau<t$ and $x,\xi\in\mathbb{R}^{d}$, cf. \citep{Stroock and Varadhan 1979}.

Since the diffusion coefficient $a(x,t)=(a^{ij}(x,t))$ is uniformly
elliptic and $b(x,t)$ is bounded, so that $p_{a;b}(\tau,\xi;t,x)$
is strictly positive and is locally Hölder's continuous in $t>\tau\geq0$,
and $\xi,x\in\mathbb{R}^{d}$. The conditional laws $\mathbb{P}_{a;b}^{\xi,\tau}\left[\cdot|X_{T}=\eta\right]$
may be constructed point-wise in $(\xi,\eta)$, which may be described
as the following.

Let $T>\tau\geq0$ be arbitrary but fixed. According to \citep[(14.1) on page 161]{Dellacherie and Meyer Volume D},
the conditional law (or called the pinned measure) $\mathbb{P}_{a;b}^{\xi,\tau\rightarrow\eta,T}$
is the unique probability measure on $\varOmega$ such that 
\[
\mathbb{P}_{a;b}^{\xi,\tau\rightarrow\eta,T}\left[X_{s}=\xi\textrm{ for }s\leq\tau\textrm{ and }X_{t}=\eta\textrm{ for }t\geq T\right]=1,
\]
and it possesses Markov property with a transition probability density
function 
\begin{equation}
q_{a;b}(s,x;t,y)=\frac{p_{a;b}(s,x;t,y)p_{a;b}(t,y;T,\eta)}{p_{a;b}(s,x;T,\eta)}\label{q-transit}
\end{equation}
for $\tau\leq s<t\leq T$.

Without losing the generality, we may set $\tau=0$ in our discussion
below. The finite dimensional marginal distribution 
\[
\mathbb{P}_{a;b}^{\xi,0\rightarrow\eta,T}\left[X_{t_{1}}\in\textrm{d}x_{1},\cdots,X_{t_{k}}\in\textrm{d}x_{k}\right],
\]
where $0=t_{0}<t_{1}<\cdots<t_{k}<t_{k+1}=T$, is the measure on $\mathbb{R}^{d\times k}$
with its Borel $\sigma$-algebra, given by 
\begin{equation}
\frac{p_{a;b}(t_{k},x_{k};T,\eta)}{p_{a;b}(0,\xi;T,\eta)}p_{a;b}(0,\xi;t_{1},x_{1})p_{a;b}(t_{1},x_{1};t_{2},x_{2})\cdots p_{a;b}(t_{k-1},x_{k-1};t_{k},x_{k})\textrm{d}x_{1}\cdots\textrm{d}x_{k}.\label{c-d--2as}
\end{equation}
Therefore $\mathbb{P}_{a;b}^{\xi,0\rightarrow\eta,T}$ is absolutely
continuous with respect to the diffusion measure $\mathbb{P}_{a;b}^{\xi,0}$
on $\mathcal{F}_{t}^{0}$ and 
\begin{equation}
\left.\frac{\textrm{d}\mathbb{P}_{a;b}^{\xi,0\rightarrow\eta,T}}{\textrm{d}\mathbb{P}_{a;b}^{\xi,0}}\right|_{\mathcal{F}_{t}^{0}}=\frac{p_{a;b}(t,X_{t};T,\eta)}{p_{a;b}(0,\xi;T,\eta)}\label{des-01}
\end{equation}
for $0\leq t<T$. It follows that for every pair $\eta\in\mathbb{R}^{d}$
and $T>0$, the diffusion bridge measure $\mathbb{P}_{a;b}^{\xi,0\rightarrow\eta,T}$
is a diffusion with its infinitesimal generator 
\begin{equation}
\mathscr{L}_{a;b}^{(\eta,T)}=\mathscr{L}_{a;b}+2\sum_{i,k}a^{ki}(x,t)\frac{\partial}{\partial x^{k}}\ln p_{a;b}(t,x;T,\eta)\frac{\partial}{\partial x^{i}}\label{pinned-02}
\end{equation}
and its support $\varOmega_{T}\equiv C([0,T];\mathbb{R}^{d})$ which
is embedded canonically in $\varOmega$.

It is important to notice that the splitting integration formula holds:
\begin{align}
\mathbb{P}_{a;b}^{\xi,0}\left[F\right] & =\int_{\mathbb{R}^{d}}\mathbb{P}_{a;b}^{\xi,0}\left[F|X_{T}=\eta\right]\mathbb{P}_{a;b}^{\xi,0}\left[X_{T}\in\textrm{d}\eta\right]\nonumber \\
 & =\int_{\mathbb{R}^{d}}\mathbb{P}_{a;b}^{\xi,0\rightarrow\eta,T}\left[F\right]p_{a;b}(0,\xi;T,\eta)\textrm{d}\eta\label{cd-s1}
\end{align}
for any integrable $F\in\mathcal{F}_{T}^{0}$.

For every $T>0$, $\tau_{T}$ denotes the time inverse operation on
$\varOmega_{T}$ which sends a path $x(t)$ to $x((T-t)^{+})$.

We are now in the position to state a key result in this paper, which
is about the duality of conditional laws of diffusion processes. 
\begin{thm}
\label{thm2.2}Under the assumptions on $a$ and $b,$and assume that
$\nabla\cdot b$ is bounded. Let $T>0$ and $\xi,\eta\in\mathbb{R}^{d}$.
Then the probability measures $\mathbb{P}_{a^{T};-b^{T}}^{\eta,0\rightarrow\xi,T}\circ\tau_{T}$
and $\mathbb{P}_{a;b}^{\xi,0\rightarrow\eta,T}$ are absolutely continuous
with each other, and 
\begin{equation}
\frac{\textrm{d}\mathbb{P}_{a^{T};-b^{T}}^{\eta,0\rightarrow\xi,T}\circ\tau_{T}}{\textrm{d}\mathbb{P}_{a;b}^{\xi,0\rightarrow\eta,T}}=\frac{\textrm{e}^{\int_{0}^{T}\nabla\cdot b(X_{s},s)\textrm{d}s}}{\mathbb{P}_{a;b}^{\xi,0\rightarrow\eta,T}\left[\textrm{e}^{\int_{0}^{T}\nabla\cdot b(X_{s},s)\textrm{d}s}\right]},\label{Dual-D1}
\end{equation}
where $X=(X_{t})_{t\geq0}$ denotes the coordinate process on $\varOmega$. 
\end{thm}

To prove the duality (\ref{Dual-D1}) we first establish several facts
about the measure $\mathbb{P}_{a^{T};-b^{T}}^{\eta,0\rightarrow\xi,T}\circ\tau_{T}$. 
\begin{lem}
\textup{\label{lem2.3} For a partition} $0=t_{0}<t_{1}<\ldots<t_{k}<t_{k+1}=T$\textup{,
the finite dimensional distribution
\[
\mathbb{P}_{a^{T};-b^{T}}^{\eta,0\rightarrow\xi,T}\circ\tau_{T}\left[X_{t_{1}}\in\textrm{d}x_{k},\cdots,X_{t_{k}}\in\textrm{d}x_{1}\right]
\]
}equals 
\begin{equation}
\frac{\varSigma_{a;b,\nabla\cdot b}^{\star}(x_{k},t_{k};\eta,T)}{\varSigma_{a;b,\nabla\cdot b}^{\star}(\xi,0;\eta,T)}\prod_{j=1}^{k}\varSigma_{a;b,\nabla\cdot b}^{\star}(x_{j-1},t_{j-1};x_{j},t_{j})\textrm{d}x_{1}\cdots\textrm{d}x_{k}\label{M-DE01}
\end{equation}
where $x_{0}=\xi$ and $x_{k+1}=\eta$. 
\end{lem}

\begin{proof}
The finite dimensional distribution by definition 
\[
\mathbb{P}_{a^{T};-b^{T}}^{\eta,0\rightarrow\xi,T}\circ\tau_{T}\left[X_{t_{1}}\in\textrm{d}x_{1},\cdots,X_{t_{k}}\in\textrm{d}x_{k}\right]
\]
equals 
\[
\mathbb{P}_{a^{T};-b^{T}}^{\eta,0\rightarrow\xi,T}\left[X_{T-t_{k}}\in\textrm{d}x_{k},\cdots,X_{T-t_{1}}\in\textrm{d}x_{1}\right]
\]
which is absolutely continuous with respect to the Lebesgue measure
on $\mathbb{R}^{d\times k}$ given by 
\begin{equation}
\frac{p_{a^{T},-b^{T}}(T-t_{1},x_{k};T,\xi)}{p_{a^{T},-b^{T}}(0,\eta;T,\xi)}\prod_{j=1}^{k}p_{a^{T},-b^{T}}(T-t_{j+1},x_{k-j};T-t_{j},x_{k+1-j})\textrm{d}x_{1}\cdots\textrm{d}x_{k}\label{h-c-2-1}
\end{equation}
Using the relation that 
\[
p_{a^{T};-b^{T}}(T-t,y;T-s,x)=\varSigma_{a;b,\nabla\cdot b}^{\star}(x,s;y,t)
\]
for any $0\leq s<t\leq T$, and making change variable $x_{i}$ to
$x_{k+1-i}$ ($i=1,\cdots,k$), the equality (\ref{M-DE01}) follows
immediately. 
\end{proof}
The second tool we need for proving the conditional law duality is
the Feynman-Kac formula for solutions to the backward problem of linear
parabolic equations. We formulate a general version as we will need
this in dealing with compressible fluid flows.

Let $\varPhi=(\varPhi^{i})_{i=1,\ldots,n}$ be a solution to the (backward
problem of) parabolic system 
\begin{equation}
(\mathscr{L}_{a;b}+\partial_{t})\varPhi^{i}-\sum_{j=1}^{n}\varTheta_{j}^{i}\varPhi^{j}+F^{i}=0\quad\textrm{ in }\mathbb{R}^{d}\times[0,\infty),\label{bp-fey-01}
\end{equation}
where $i=1,\ldots,n$, and $\varTheta_{j}^{i}$ and $F^{i}$ are bounded
Borel measurable functions.

For each continuous path $\psi\in C([0,\infty),\mathbb{R}^{d})$ and
every $\tau\geq0$, $t\mapsto Q(\tau,\psi;t)=(Q_{j}^{i}(\tau,\psi;t))_{i,j\leq n}$
denotes the unique solution of the linear system of ordinary differential
equations: 
\begin{equation}
\frac{\textrm{d}}{\textrm{d}t}Q_{j}^{i}(\tau,\psi;t)+Q_{k}^{i}(\tau,\psi;t)\varTheta_{j}^{k}(\psi(t),t)=0,\quad Q_{j}^{i}(\tau,\psi;\tau)=\delta_{j}^{i}.\label{gu-01}
\end{equation}

The following result is the well-known Feynman-Kac formula. 
\begin{lem}
\label{lem5.1} Suppose $\varPhi(x,t)$ is a $C^{2,1}$ solution to
\eqref{bp-fey-01} with bounded first derivatives. Then 
\begin{equation}
\varPhi^{i}(\xi,\tau)=\mathbb{P}_{a;b}^{\xi,\tau}\left[Q_{j}^{i}(\tau,X;t)\varPhi^{j}(X_{t},t)\right]+\mathbb{P}_{a;b}^{\xi,\tau}\left[\int_{\tau}^{t}Q_{j}^{i}(\tau,X;s)F^{j}(X_{s},s)\textrm{d}s\right]\label{Fey-01}
\end{equation}
where $i=1,\ldots,n$, $t>\tau$ and $\xi\in\mathbb{R}^{d}$. 
\end{lem}

The representation follows easily from an application of the It\^o
formula, see for example Freidlin \citep{Freidlin1985}. As a consequence,
we have the following 
\begin{lem}
\label{lem5.2} Under the assumptions on $a$ and $b$, and suppose
$c(x,t)$ is bounded and continuous. Then 
\begin{equation}
\frac{\varSigma_{a;b,c}^{\star}(\xi,\tau;x,t)}{p_{a;b}(\tau,\xi,t,x)}=\mathbb{P}_{a;b}^{\xi,\tau\rightarrow x,t}\left[\textrm{e}^{\int_{\tau}^{t}c(X_{s},s)\textrm{d}s}\right]\label{rep-g1}
\end{equation}
for $x,\xi\in\mathbb{R}^{d}$ and $t>\tau\geq0$. 
\end{lem}

\begin{proof}
Let $T>0$ be arbitrary but fixed, and $f$ be a $C^{2,1}$ bounded
solution to the backward parabolic equation: 
\[
\left(\mathscr{L}_{a;b,c}+\partial_{t}\right)f=0
\]
such that $f(x,t)\rightarrow f_{0}(x)$ as $t\uparrow T$. The equation
may be rewritten as the following: 
\[
\left(\mathscr{L}_{a;b}+\partial_{t}\right)f(x,t)+c(x,t)f(x,t)=0
\]
so that, according Lemma \ref{lem5.1} 
\begin{equation}
f(\xi,\tau)=\mathbb{P}_{a;b}^{\xi,\tau}\left[C(\tau,X;T)f(X_{T},T)\right]\label{rep-f01}
\end{equation}
for $0\leq\tau<T$, where $C$ solves the ordinary equation 
\[
\frac{\textrm{d}}{\textrm{d}t}C(\tau,X;t)-C(\tau,X;t)c(X_{t},t)=0,\quad C(\tau,X;\tau)=1
\]
which has the unique solution 
\[
C(\tau,X;t)=\textrm{e}^{\int_{\tau}^{t}c(X_{s},s)\textrm{d}s}
\]
for $t\geq\tau$. By taking condition that $X_{T}=x$ we may rewrite
the equality (\ref{rep-f01}) as the following 
\[
f(\xi,\tau)=\int_{\mathbb{R}^{d}}p_{a;b}(\tau,\xi;T,x)\mathbb{P}_{a;b}^{\xi,\tau}\left[\left.\textrm{e}^{\int_{\tau}^{T}c(X_{s},s)\textrm{d}s}\right|X_{T}=x\right]f(x,T)\textrm{d}x,
\]
which by definition yields 
\[
\varSigma_{a;b,c}^{\star}(\xi,\tau;x,T)=p_{a;b}(\tau,\xi;T,x)\mathbb{P}_{a;b}^{\xi,\tau\rightarrow x,T}\left[\textrm{e}^{\int_{\tau}^{T}c(X_{s},s)\textrm{d}s}\right].
\]
The proof is complete. 
\end{proof}
\begin{lem}
\label{lem2.6}Let $D:0=t_{0}<t_{1}<\ldots<t_{k}<t_{k+1}=T$. Then
\[
\mathbb{P}_{a^{T};-b^{T}}^{\eta,0\rightarrow\xi,T}\circ\tau_{T}\left[X_{t_{1}}\in\textrm{d}x_{k},\cdots,X_{t_{k}}\in\textrm{d}x_{1}\right]=H_{D}(x_{0},\cdots,x_{k+1})\mathbb{P}_{a;b}^{\xi,0\rightarrow\eta,T}\left[X_{t_{1}}\in\textrm{d}x_{1},\cdots,X_{t_{k}}\in\textrm{d}x_{k}\right]
\]
where 
\[
H_{D}(x_{0},\cdots,x_{k+1})=\frac{1}{\mathbb{P}_{a;b}^{\xi,0\rightarrow\eta,T}\left[\textrm{e}^{\int_{0}^{T}\nabla\cdot b(X_{s},s)\textrm{d}s}\right]}\prod_{i=1}^{k+1}\mathbb{P}_{a;b}^{x_{i-1},t_{i-1}\rightarrow x_{i},t_{i}}\left[\textrm{e}^{\int_{t_{i-1}}^{t_{i}}\nabla\cdot b(X_{s},s)\textrm{d}s}\right].
\]
and $x_{0}=\xi$ and $x_{k+1}=\eta$. 
\end{lem}

We are now in a position to prove the duality of the conditional laws. 
\begin{proof}
(\emph{of Theorem} \ref{thm2.2}) From the previous lemma we can see
that the probability measure $\mathbb{P}_{a^{T};-b^{T}}^{\eta,0\rightarrow\xi,T}\circ\tau_{T}$
has Markov property, so we only compute the one-dimensional marginal
distribution of $X_{t}$ under $\mathbb{P}_{a^{T};-b^{T}}^{\eta,0\rightarrow\xi,T}\circ\tau_{T}$.
Suppose $f$ is bounded and Borel measurable, let us compute 
\[
I=\mathbb{P}_{a^{T};-b^{T}}^{\eta,0\rightarrow\xi,T}\circ\tau_{T}\left[f(X_{t})\right]=\mathbb{P}_{a^{T};-b^{T}}^{\eta,0\rightarrow\xi,T}\left[f(X_{T-t})\right].
\]
According to Lemma \ref{lem2.6} and the fact that 
\[
\left.\frac{\textrm{d}\mathbb{P}_{a;b}^{\xi,0\rightarrow\eta,T}}{\textrm{d}\mathbb{P}_{a;b}^{\xi,0}}\right|_{\mathcal{F}_{t}^{0}}=\frac{p_{a;b}(t,X_{t};T,\eta)}{p_{a;b}(0,\xi;T,\eta)}\quad\textrm{ for }0\leq t<T,
\]
we may write 
\begin{align*}
I & =\frac{1}{\mathbb{P}_{a;b}^{\xi,0\rightarrow\eta,T}\left[\textrm{e}^{\int_{0}^{T}\nabla\cdot b(X_{s},s)\textrm{d}s}\right]}\\
 & \cdot\mathbb{P}_{a;b}^{\xi,0\rightarrow\eta,T}\left[\lim_{\varepsilon\downarrow0}\frac{f(X_{t})\textrm{e}^{\int_{0}^{t}\nabla\cdot b(X_{s},s)\textrm{d}s}}{p_{a;b}(t,X_{t};T,\eta)}\mathbb{P}_{a;b}^{X_{t},t}\left[p_{a;b}(T-\varepsilon,X_{T-\varepsilon};T,\eta)\textrm{e}^{\int_{t}^{T-\varepsilon}\nabla\cdot b(X_{s},s)\textrm{d}s}\right]\right]\\
 & =\frac{1}{\mathbb{P}_{a;b}^{\xi,0\rightarrow\eta,T}\left[\textrm{e}^{\int_{0}^{T}\nabla\cdot b(X_{s},s)\textrm{d}s}\right]}\\
 & \cdot\lim_{\varepsilon\downarrow0}\mathbb{P}_{a;b}^{\xi,0}\left[\mathbb{P}_{a;b}^{X_{t},t}\left(\frac{p_{a;b}(T-\varepsilon,X_{T-\varepsilon};T,\eta)}{p_{a;b}(0,\xi;T,\eta)}\textrm{e}^{\int_{t}^{T-\varepsilon}\nabla\cdot b(X_{s},s)\textrm{d}s}\right)\textrm{e}^{\int_{0}^{t}\nabla\cdot b(X_{s},s)\textrm{d}s}f(X_{t})\right]\\
 & =\frac{1}{\mathbb{P}_{a;b}^{\xi,0\rightarrow\eta,T}\left[\textrm{e}^{\int_{0}^{T}\nabla\cdot b(X_{s},s)\textrm{d}s}\right]}\\
 & \cdot\lim_{\varepsilon\downarrow0}\mathbb{P}_{a;b}^{\xi,0}\left[\frac{p_{a;b}(T-\varepsilon,X_{T-\varepsilon};T,\eta)}{p_{a;b}(0,\xi;T,\eta)}\textrm{e}^{\int_{0}^{T-\varepsilon}\nabla\cdot b(X_{s},s)\textrm{d}s}f(X_{t})\right]\\
 & =\mathbb{P}_{a;b}^{\xi,0\rightarrow\eta,T}\left\{ \frac{\textrm{e}^{\int_{0}^{T}\nabla\cdot b(X_{s},s)\textrm{d}s}}{\mathbb{P}_{a;b}^{\xi,0\rightarrow\eta,T}\left[\textrm{e}^{\int_{0}^{T}\nabla\cdot b(X_{s},s)\textrm{d}s}\right]}f(X_{t})\right\} 
\end{align*}
where the third equality follows from the Markov property (the general
form), which in turn implies that 
\[
\frac{\textrm{d}\mathbb{P}_{a^{T};-b^{T}}^{\eta,0\rightarrow\xi,T}\circ\tau_{T}}{\textrm{d}\mathbb{P}_{a;b}^{\xi,0\rightarrow\eta,T}}=\frac{\textrm{e}^{\int_{0}^{T}\nabla\cdot b(X_{s},s)\textrm{d}s}}{\mathbb{P}_{a;b}^{\xi,0\rightarrow\eta,T}\left[\textrm{e}^{\int_{0}^{T}\nabla\cdot b(X_{s},s)\textrm{d}s}\right]}
\]
and the proof is complete. 
\end{proof}

\section{Functional integral representations for parabolic equations}

In this section we establish functional integral representation for
solutions to the initial value problem of the following system of
linear parabolic equations 
\begin{equation}
\left(\mathscr{L}_{a;b,c}-\partial_{t}\right)\varPsi^{i}+\sum_{j=1}^{n}q_{j}^{i}\varPsi^{j}+f^{i}=0\label{01-t}
\end{equation}
subject to the initial value $\varPsi^{i}(x,0)=\varPsi_{0}^{i}(x)$,
where $i=1,\ldots,n$. The multiplier $q(x,t)=(q^{ij}(x,t))$ is an
$n\times n$ square-matrix valued bounded and continuous function. 
\begin{thm}
\label{thm3.1}For every $T>0$ and $\psi\in C([0,T];\mathbb{R}^{d})$,
$t\mapsto A(\psi,T;t)$ denotes the solution to the ordinary linear
differential equations 
\begin{equation}
\frac{\textrm{d}}{\textrm{d}t}A_{j}^{i}(\psi,T;t)=-A_{k}^{i}(\psi,T;t)q_{j}^{k}(\psi(t),t)-A_{j}^{i}(\psi,T;t)c(\psi(t),t)\label{A-1}
\end{equation}
and 
\begin{equation}
A_{j}^{i}(\psi,T;T)=\delta_{ij}.\label{A-2}
\end{equation}
Then the following functional integral representation holds 
\begin{align*}
\varPsi^{i}(\xi,T) & =\int_{\mathbb{R}^{d}}\mathbb{P}_{a;-b}^{\eta,0\rightarrow\xi,T}\left[\textrm{e}^{-\int_{0}^{T}\nabla\cdot b(X_{s},s)\textrm{d}s}A_{j}^{i}(X,T;0)\varPsi_{0}^{j}(\eta)\right]p_{a;-b}(0,\eta;T,\xi)\textrm{d}\eta\\
 & +\int_{0}^{T}\int_{\mathbb{R}^{d}}\mathbb{P}_{a;-b}^{\eta,0\rightarrow\xi,T}\left[\textrm{e}^{-\int_{0}^{T}\nabla\cdot b(X_{s},s)\textrm{d}s}A_{j}^{i}(X,T;s)f^{j}(X_{s},s)\right]p_{a;-b}(0,\eta;T,\xi)\textrm{d}\eta\textrm{d}s
\end{align*}
for any $T>0$ and $\xi\in\mathbb{R}^{d}$. 
\end{thm}

\begin{proof}
Let $T>0$ be arbitrary but fixed in the discussion below. Let $\varPhi(x,t)=\varPsi(x,(T-t)^{+})$.
Then 
\begin{equation}
\left(\mathscr{L}_{a^{T};b^{T}}+\partial_{t}\right)\varPhi^{i}-\sum_{j=1}^{n}\varTheta_{j}^{i}\varPsi^{j}+F^{i}=0\quad\textrm{ in }\mathbb{R}^{d}\times[0,T],\label{F-BE-01}
\end{equation}
where $\varTheta_{j}^{i}(x,t)=-q_{j}^{i}(x,(T-t)^{+})-c(x,(T-t)^{+})\delta_{ij}$
and $F^{i}(x,t)=f^{i}(x,(T-t)^{+})$. By Lemma \ref{lem5.1} 
\begin{equation}
\varPsi^{i}(\xi,T)=\mathbb{P}_{a^{T};b^{T}}^{\xi,0}\left[Q_{j}^{i}(X,T;T)\varPsi_{0}^{j}(X_{T})\right]+\int_{0}^{T}\mathbb{P}_{a^{T};b^{T}}^{\xi,0}\left[Q_{j}^{i}(X,T;s)f^{j}(X_{s},T-s)\right]\textrm{d}s\label{Fey-01-1}
\end{equation}
where $i=1,\ldots,n$, $\xi\in\mathbb{R}^{d}$, and $Q_{j}^{i}(\psi,T;t)$
are the solutions to 
\begin{equation}
\frac{\textrm{d}}{\textrm{d}t}Q_{j}^{i}(\psi,T;t)+Q_{k}^{i}(\psi,T;t)\varTheta_{j}^{k}(\psi(t),t)=0,\quad Q_{j}^{i}(\psi,T;0)=\delta_{j}^{i}.\label{gu-01-1}
\end{equation}
We next rewrite the expectations using conditional laws to obtain
that 
\begin{align*}
\varPsi^{i}(\xi,T) & =\int_{\mathbb{R}^{d}}\mathbb{P}_{a^{T};b^{T}}^{\xi,0\rightarrow\eta,T}\left[Q_{j}^{i}(X,T;T)\right]\varPsi_{0}^{j}(\eta)p_{a^{T};b^{T}}(0,\xi;T,\eta)\textrm{d}\eta\\
 & +\int_{\mathbb{R}^{d}}\int_{0}^{T}\mathbb{P}_{a^{T};b^{T}}^{\xi,0\rightarrow\eta,T}\left[Q_{j}^{i}(X,T;s)f^{j}(X_{s},T-s)\right]p_{a^{T};b^{T}}(0,\xi;T,\eta)\textrm{d}\eta\textrm{d}s\\
 & =\int_{\mathbb{R}^{d}}\mathbb{P}_{a^{T};b^{T}}^{\xi,0\rightarrow\eta,T}\left[Q_{j}^{i}(X,T;T)\right]\varPsi_{0}^{j}(\eta)p_{a^{T};b^{T}}(0,\xi;T,\eta)\textrm{d}\eta\\
 & +\int_{0}^{T}\int_{\mathbb{R}^{d}}\mathbb{P}_{a^{T};b^{T}}^{\xi,0\rightarrow\eta,T}\left[Q_{j}^{i}(X,T;T-s)f^{j}(X_{s},s)\right]p_{a^{T};b^{T}}(0,\xi;T,\eta)\textrm{d}\eta\textrm{d}s.
\end{align*}
While, according to 
\begin{align*}
p_{a^{T};b^{T}}(0,\xi;T,\eta) & =\varSigma_{a;-b,-\nabla\cdot b}^{\star}(\eta,0;\xi,T)\\
 & =p_{a;-b}(0,\eta;T,\xi)\mathbb{P}_{a;-b}^{\eta,0\rightarrow\xi,T}\left[\textrm{e}^{-\int_{0}^{T}\nabla\cdot b(X_{s},s)\textrm{d}s}\right]
\end{align*}
so by substituting this into the previous representation we obtain
\begin{align*}
\varPsi^{i}(\xi,T) & =\int_{\mathbb{R}^{d}}\mathbb{P}_{a^{T};b^{T}}^{\xi,0\rightarrow\eta,T}\left[Q_{j}^{i}(X,T;T)\varPsi_{0}^{j}(\eta)\right]\mathbb{P}_{a;-b}^{\eta,0\rightarrow\xi,T}\left[\textrm{e}^{-\int_{0}^{T}\nabla\cdot b(X_{s},s)\textrm{d}s}\right]p_{a;-b}(0,\eta;T,\xi)\textrm{d}\eta\\
 & +\int_{0}^{T}\int_{\mathbb{R}^{d}}\mathbb{P}_{a^{T};b^{T}}^{\xi,0\rightarrow\eta,T}\left[Q_{j}^{i}(X,T;s)f^{j}(X_{s},T-s)\right]\\
 & \quad\cdot\mathbb{P}_{a;-b}^{\eta,0\rightarrow\xi,T}\left[\textrm{e}^{-\int_{0}^{T}\nabla\cdot b(X_{s},s)\textrm{d}s}\right]p_{a;-b}(0,\eta;T,\xi)\textrm{d}\eta\textrm{d}s.
\end{align*}
Finally by using the conditional law duality, Theorem \ref{thm2.2},
\begin{equation}
\frac{\textrm{d}\mathbb{P}_{a^{T};b^{T}}^{\xi,0\rightarrow\eta,T}\circ\tau_{T}}{\textrm{d}\mathbb{P}_{a;-b}^{\eta,0\rightarrow\xi,T}}=\frac{\textrm{e}^{-\int_{0}^{T}\nabla\cdot b(X_{s},s)\textrm{d}s}}{\mathbb{P}_{a;-b}^{\eta,0\rightarrow\xi,T}\left[\textrm{e}^{-\int_{0}^{T}\nabla\cdot b(X_{s},s)\textrm{d}s}\right]},\label{Dual-D1-1}
\end{equation}
we therefore conclude that 
\begin{align*}
\varPsi^{i}(\xi,T) & =\int_{\mathbb{R}^{d}}\mathbb{P}_{a;-b}^{\eta,0\rightarrow\xi,T}\left[\textrm{e}^{-\int_{0}^{T}\nabla\cdot b(X_{s},s)\textrm{d}s}Q_{j}^{i}(X\circ\tau_{T},T;T)\varPsi_{0}^{j}(\eta)\right]p_{a;-b}(0,\eta;T,\xi)\textrm{d}\eta\\
 & +\int_{0}^{T}\int_{\mathbb{R}^{d}}\mathbb{P}_{a;-b}^{\eta,0\rightarrow\xi,T}\left[\textrm{e}^{-\int_{0}^{T}\nabla\cdot b(X_{s},s)\textrm{d}s}Q_{j}^{i}(X\circ\tau_{T},T;s)f^{j}(X_{T-s},T-s)\right]\\
 & \quad\cdot p_{a;-b}(0,\eta;T,\xi)\textrm{d}\eta\textrm{d}s\\
 & =\int_{\mathbb{R}^{d}}\mathbb{P}_{a;-b}^{\eta,0\rightarrow\xi,T}\left[\textrm{e}^{-\int_{0}^{T}\nabla\cdot b(X_{s},s)\textrm{d}s}Q_{j}^{i}(X\circ\tau_{T},T;T)\varPsi_{0}^{j}(\eta)\right]p_{a;-b}(0,\eta;T,\xi)\textrm{d}\eta\\
 & +\int_{0}^{T}\int_{\mathbb{R}^{d}}\mathbb{P}_{a;-b}^{\eta,0\rightarrow\xi,T}\left[\textrm{e}^{-\int_{0}^{T}\nabla\cdot b(X_{s},s)\textrm{d}s}Q_{j}^{i}(X\circ\tau_{T},T;T-s)f^{j}(X_{s},s)\right]\\
 & \quad\cdot p_{a;-b}(0,\eta;T,\xi)\textrm{d}\eta\textrm{d}s.
\end{align*}
It remains to verify $A_{j}^{i}(\psi,T;t)=Q_{j}^{i}(X\circ\tau_{T},T;T-t)$
(for $t\in[0,T]$) are the solutions to (\ref{A-1}, \ref{A-2}).
By definition $A_{j}^{i}(\psi,T;T)=\delta_{j}^{i}$. While by definition
\begin{align*}
Q_{j}^{i}(\psi\circ\tau_{T},T;T-t) & =\delta_{ij}-\int_{0}^{T-t}Q_{k}^{i}(\psi\circ\tau_{T},T;s)\varTheta_{j}^{k}(\psi(T-s),s)\textrm{d}s\\
 & =\delta_{ij}+\int_{T}^{t}Q_{k}^{i}(\psi\circ\tau_{T},T;T-s)\varTheta_{j}^{k}(\psi(s),T-s)\textrm{d}s.
\end{align*}
That is $A(\psi,T;t)$ solves the integral equation 
\[
A(\psi,T;t)=\delta_{ij}+\int_{T}^{t}A_{k}^{i}(\psi,T;s)\varTheta_{j}^{k}(\psi(s),T-s)\textrm{d}s
\]
which yields the conclusion. The proof is complete.
\end{proof}
\vskip0.5truecm

\emph{Acknowledgement}. This work is supported partially by the EPSRC
Centre for Doctoral Training in Mathematics of Random Systems: Analysis,
Modelling and Simulation (EP/S023925/1).

\end{document}